\documentclass[12pt,reqno]{elsarticle}

\usepackage{amsfonts,amsmath,amsthm}
\usepackage{amssymb,epsfig,ascmac}
\usepackage{cases}
\usepackage[usenames]{color}
\usepackage{enumerate} 

\usepackage{color} 
\definecolor{vert}{rgb}{0,0.6,0}

 \theoremstyle{plain}
 \begingroup
 \newtheorem{thm}{Theorem}[section]
  
 \newtheorem{lem}[thm]{Lemma}
 \newtheorem{prop}[thm]{Proposition}
 
 \endgroup

 \theoremstyle{definition}
 \begingroup

 \endgroup

 \begingroup
 \newtheorem{rem}[thm]{Remark}
 \endgroup

 \numberwithin{equation}{section}

\newcommand{\C}{\mathbb{C}}
\newcommand{\E}{\mathbb{E}}

\newcommand{\N}{\mathbb{N}}
\newcommand{\bP}{\mathbb{P}}
\newcommand{\R}{\mathbb{R}}

\newcommand{\T}{\mathbb{T}}

\newcommand{\cM}{\mathcal{M}}

\newcommand{\cE}{\mathcal{E}}

\newcommand{\AC}{{\rm AC\,}}

\newcommand{\Li}{L^{\infty}}
\newcommand{\W}{W^{1,\infty}}

\newcommand{\Q}{\mathbb{T}^{n}\times(0,\infty)}

\newcommand{\cQ}{\mathbb{T}^{n}\times[0,\infty)}

\newcommand{\al}{\alpha}
\newcommand{\gam}{\gamma}
\newcommand{\del}{\delta}
\newcommand{\ep}{\varepsilon}

\newcommand{\lam}{\lambda}

\newcommand{\om}{\omega}

\newcommand{\Del}{\Delta}

\newcommand{\ol}{\overline}
\newcommand{\ul}{\underline}
\newcommand{\pl}{\partial}
\newcommand{\supp}{{\rm supp}\,}
\newcommand{\inter}{{\rm int}\,}

\newcommand{\co}{{\rm co}\,}

\newcommand{\dom}{{\rm dom}\,}

\begin{document}
\title
{A dynamical approach to the large-time behavior of solutions to weakly coupled systems of Hamilton--Jacobi equations}


\author{H. Mitake}
\ead{mitake@math.sci.fukuoka-u.ac.jp}
\address{Department of Applied Mathematics, 
Faculty of Science, Fukuoka University, 
Fukuoka 814-0180, Japan}

\author{H. V. Tran}
\ead{tvhung@math.berkeley.edu}
\address{Department of Mathematics, 
The University of Chicago, 5734 S. University Avenue Chicago, Illinois 60637, USA}

\begin{keyword}
Large-time Behavior; Hamilton--Jacobi Equations; 
Weakly Coupled Systems; Ergodic Problem; Switching Cost Problems; Piecewise-deterministic Markov processes; Viscosity Solutions. 
\MSC[2010]{
35B40 \sep 
35F55 \sep 
49L25 
}
\end{keyword}

\date{\today}

\begin{abstract}
We investigate the large-time behavior of the value functions 
of the optimal control problems on the $n$-dimensional torus which appear 
in the dynamic programming for the system whose states are governed 
by random changes. 
From the point of view of the study on partial differential equations, 
it is equivalent to consider viscosity solutions of quasi-monotone 
weakly coupled systems of Hamilton--Jacobi equations. 
The large-time behavior of viscosity solutions of this problem 
has been recently studied by 
{the authors and 
Camilli, Ley, Loreti, and Nguyen} for some special cases, independently, 
but the general cases remain widely open.  
We establish a convergence result to asymptotic solutions as time goes 
to infinity under rather general assumptions by using dynamical 
properties of value functions. \\
\bigskip

\noindent
\textbf{R\'esum\'e}\\
Nous \'etudions le comportement en temps grand des fonctions valeur associ\'es aux probl\`emes de contr\^ole optimal sur le  tore $n$-dimensionnel qui apparaissent dans le cadre de la programmation dynamique des syst\`emes dont les \'etats sont gouvern\'es par des changements al\'eatoires. 
Du point de vue de l'\'etude des \'equations aux d\'eriv\'ees partielles, il est \'equivalent de consid\'erer des solutions de viscosit\'e des syst\`emes de Hamilton--Jacobi quasi-monotones faiblement coupl\'es. 
Le comportement en temps grand des solutions de viscosit\'e de ce probl\`eme a \'et\'e r\'ecemment \'etudi\'e par les auteurs, ainsi que par Camilli, Ley, Loreti, et Nguyen pour certains cas particuliers, 
de fa\c{c}on ind\'{e}pendante, mais les cas g\'en\'eraux restent largement ouverts. 
Nous \'etablissons un r\'esultat de convergence asymptotique des solutions sous des hypoth\`eses assez g\'en\'erales,  en utilisant des propri\'et\'es dynamiques des fonctions valeur. 
\end{abstract}

\maketitle



\section{Introduction and Main Result}
In this paper we deal with optimal control problems, or calculus
of variations, 
which appear in the dynamic programming for the system 
whose states are governed by random changes. 
More precisely, we consider the minimizing problem: 
\begin{equation}\label{def-value}
\text{Minimize} \quad 
\E_{i}\Big[\int_{-t}^{0}
L_{\nu(s)}(\gam(s),\dot{\gam}(s))\,ds
+g_{\nu(-t)}(\gam(-t))\Big], 
\end{equation}
over all controls $\gam\in\AC([-t,0])$ with $\gam(0)=x$ 
for any fixed $(x,t)\in\cQ$, 
where the Lagrangians $L_{i}(x,q):\T^n \times \R^n\to\R$ are 
derived from the Fenchel-Legendre transforms of given Hamiltonians 
$H_i$ and we denote by $\AC([-t,0])$ the set of 
absolutely continuous functions 
on $[-t,0]$ with values in $\T^n$.  
The functions $g_{i}$ are given real-valued continuous functions on $\T^n$ for $i=1,2$. 
Here $\E_{i}$ denotes the expectation of 
a process with $\nu(0)=i$, where 
$\nu$  is a $\{1,2\}$-valued process which 
is a continuous-time Markov chain 
on $(-\infty,0]$
(notice that time is reversed)
such that 
for $s\le 0$, $\Del s >0$,
\begin{equation}\label{markov}
\bP\big(\nu(s-\Del s)=j\mid 
\nu(s)=i\big)=c_i \Del s+o(\Del s) \ 
\textrm{as} \ \Del s\to0 \ \textrm{for} \ i\not=j, 
\end{equation}
where $c_i$ are given positive constants and 
$o:[0,\infty)\to [0,\infty)$ is a function
satisfying $o(r)/r\to0$ as $r\to0$. 
We call the minimizing costs of \eqref{def-value} the \textit{value functions} of optimal control problems \eqref{def-value}.

The purpose of this paper is to investigate the large-time behavior of the value
functions. 
From the point of view of partial differential equations 
it is equivalent to 
study that of viscosity solutions of
quasi-monotone weakly coupled systems of Hamilton--Jacobi equations 
\begin{numcases}
{\textrm{(C)} \hspace{1cm}}
(u_{1})_t + H_1(x,Du_{1}) + c_1 (u_{1}-u_{2}) = 0
& in $\Q$, \nonumber \\
(u_{2})_t + H_2(x,Du_{2}) + c_2 (u_{2}-u_{1}) = 0
& in $\Q$, \nonumber \\
u_{i}(x,0)=g_{i}(x) \ 
& 
on $\T^{n}$, 
\nonumber
\end{numcases}
where 
the Hamiltonians $H_{i}(x,p):\T^n \times \R^n\to\R$ are 
given continuous functions for $i=1,2$, 
which are assumed throughout the paper to satisfy the followings. 
\begin{itemize}
\item[{\rm(A1)}]
The functions $H_{i}$ are  uniformly coercive in the $p$-variable, 
i.e., 
$$
\lim_{r\to\infty}\inf\{H_{i}(x,p)\mid x\in{\T^n}, 
|p|\ge r \}=\infty.
$$
\item[{\rm(A2)}]
The functions $p\mapsto H_{i}(x,p)$ are strictly convex 
for any $x\in\T^n$. 
\end{itemize}
Here $u_i$ are real-valued unknown functions on $\T^n \times [0,\infty)$ and
$(u_i)_t=\partial u_i/\partial t$, 
$Du_i=(\partial u_i/\partial x_1,\dots, \partial u_i/\partial x_n)$ for $i=1,2$, respectively.
We are only dealing with viscosity solutions of Hamilton--Jacobi equations here and thus the term
``viscosity" will be omitted henceforth.

The existence and uniqueness results for weakly coupled systems (C)  
of Hamilton--Jacobi equations have been established by 
\cite{EL, IK}. 
In recent years, there have been many studies 
on the properties of viscosity solutions
of weakly coupled systems of Hamilton--Jacobi equations.  
See \cite{CGT2,MT1,MT2,CLL, CLLN} for instance.  
In particular, the studies on large-time behaviors were done for some special cases by the authors
\cite{MT1}, and 
Camilli, Ley, Loreti and Nguyen \cite{CLLN}, independently. 
However, the general cases remain widely open and the techniques developed 
in \cite{MT1, CLLN} are not applicable for general cases.
The coupling terms cause serious difficulties, which will be explained 
in details later.

Let us first recall the heuristic derivation of
the large-time asymptotics for (C) discussed by the authors \cite{MT1}
for readers' convenience.
We use the same notations as in \cite{MT1}. 
For simplicity, we assume that
$c_1=c_2=1$ henceforth.
Formal asymptotic expansions of the solutions
$u_1, u_2$ of (C) are considered to be of the forms 
\begin{align}
u_1(x,t)&=a_{01}(x)t+a_{11}(x)+a_{21}(x)t^{-1}+\ldots, \notag\\
u_2(x,t)&=a_{02}(x)t+a_{12}(x)+a_{22}(x)t^{-1}+\ldots \quad\notag
\end{align}
as  $t\to \infty$.
Then (C) becomes 
\begin{align}
a_{01}(x)-a_{21}(x)t^{-2}+\ldots+H_1(x,Da_{01}(x)t+Da_{11}(x)+Da_{21}(x)t^{-1}+\ldots)\notag\\
+(a_{01}(x)-a_{02}(x))t+(a_{11}(x)-a_{12}(x))+(a_{21}(x)-a_{22}(x))t^{-1}+\ldots=0, \label{heu-1}
\end{align}
and
\begin{align}
a_{02}(x)-a_{22}(x)t^{-2}+\ldots+H_2(x,Da_{02}(x)t+Da_{12}(x)+Da_{22}(x)t^{-1}+\ldots)\notag\\
+(a_{02}(x)-a_{01}(x))t+(a_{12}(x)-a_{11}(x))+(a_{22}(x)-a_{21}(x))t^{-1}+\ldots=0. \label{heu-2}
\end{align}
Sum up \eqref{heu-1} and \eqref{heu-2} to yield 
$$
H_1(x,D a_{01} t + D a_{11} + O(1/t)) + H_2(x,D a_{02} t + D a_{12} + O(1/t)) +O(1)=0
$$
as $t\to\infty$. 
Hence 
we formally get $D a_{01} = Da_{02} \equiv 0$ by the coercivity of $H_1$ and $H_2$. 
We next let $t \to \infty$ in \eqref{heu-1}, \eqref{heu-2} to achieve that 
$a_{01}(x)=a_{02}(x) \equiv a_{0}$ for some constant $a_{0} \in \R$, 
and
\begin{numcases}
{}
H_{1}(x,Da_{11}(x))+a_{11}(x)-a_{12}(x)=-a_{0} 
,\notag\\
H_{2}(x,Da_{12}(x))+a_{12}(x)-a_{11}(x)=-a_{0}
, \notag
\end{numcases}
in $\T^n$. It is then natural to study the ergodic problem
\begin{numcases}
{{\rm (E)} \hspace{1cm}}
H_{1}(x,Dv_{1}(x))+v_{1}-v_{2}= c 
& in  $\T^{n}$,\ \nonumber\\
H_{2}(x,Dv_{2}(x))+v_{2}-v_{1}= c 
& in  $\T^{n}$.  \nonumber
\end{numcases}
We here seek for a triplet 
$(v_{1},v_{2},c)\in C(\T^n)^{2}\times\R$ 
such that $(v_{1},v_{2})$ is a solution of (E). 
If $(v_{1},v_{2},c)$ is such a triplet, 
we call $(v_{1},v_{2})$ a \textit{pair of ergodic functions} 
and $c$ an \textit{ergodic constant}. 
It was proved in \cite{CGT2, MT1} that
 there exists a unique constant $c$ 
such that the ergodic problem (E)  has continuous solutions $(v_1, v_2)$.

Hence, our goal in this paper is to prove the following 
large-time asymptotics for (C). 
\begin{thm}[Main Result]\label{thm:main}
Assume that {\rm (A1), (A2)} hold.
For any $(g_{1},g_{2}) \in C(\T^n)^2$ there exists a solution
$(v_1,v_2,c) \in C(\T^n)^2 \times \R$ of {\rm(E)} such that 
\begin{equation}\label{conv}
u_{i}(x,t)+ct-v_{i}(x)\to0 \ \ \textrm{uniformly on} \ \T^{n} 
\ \text{as} \ t\to\infty 
\end{equation}
for $i=1,2$. 
\end{thm}

In the last decade, the large time behavior of solutions of single Hamilton--Jacobi equations,
\begin{equation} \label{single-HJ}
u_t+H(x,Du)=0 \ \text{in}\ \Q, 
\end{equation}
where $H$ is coercive,
 has received much attention and general convergence results for solutions have been established. 
The first general result was discovered by
Namah and Roquejoffre in \cite{NR} under the following additional assumptions:
$p \mapsto H(x,p)$ is convex, and 
\begin{equation}\label{NR-cond}
H(x,p)\ge H(x,0) 
\ \textrm{for all} \ (x,p)\in\cM\times\R^{n} 
\ \textrm{and} \ 
\max_{\cM}H(x,0)=0,   
\end{equation}
where $\cM$ is a smooth compact $n$-dimensional 
manifold without boundary. 
Then Fathi 
used dynamical system approach from weak KAM theory
in \cite{F2} to establish
the same type of convergence result,
which requires
{uniform} convexity (and smoothness) assumptions on $H(x,\cdot)$, 
i.e., $D_{pp}H(x,p)\ge\al I$ for all $(x,p)\in\cM\times\R^{n}$ 
and $\al>0$ but does not require
the specific structure \eqref{NR-cond} of Hamiltonians. 
Afterwards  Roquejoffre \cite{R},  
Davini and Siconolfi in \cite{DS}, Ishii in \cite{I2008} refined 
and generalized the approach of Fathi and 
they studied the asymptotic problem for 
Hamilton--Jacobi equations 
on $\cM$ or the whole $n$-dimensional Euclidean space. 
Besides, Barles and Souganidis \cite{BS} 
also obtained this type of results, for possibly non-convex Hamiltonians,
by using a PDE method in the context of viscosity solutions.

In the previous paper \cite{MT1},
the authors could establish Theorem \ref{thm:main}
only in two main specific cases.
In the first case, we generalized the approach in \cite{NR} and obtain
convergence result under additional assumptions similar to \eqref{NR-cond}
(see also \cite{CLLN}).
The second case is a generalization of \cite{BS} under the strong assumption that
$H_1=H_2=H$, where $H$ satisfies similar assumptions as in \cite{BS}.
We could not obtain Theorem \ref{thm:main} in its full generality 
because of the appearance of the coupling terms $u_1-u_2$ and $u_2-u_1$.

In this paper we develop a dynamical approach to 
weakly coupled systems of Hamilton--Jacobi equations 
which is inspired by the works by Davini, Siconolfi \cite{DS} 
and Ishii \cite{I2008}, and establish 
Theorem \ref{thm:main} in its full generality. The results in \cite{F2, R, DS}
can be viewed as a particular case of Theorem \ref{thm:main} when
 $H_1=H_2$, and $g_1=g_2$.
As we consider  system (C),
we need to take random switchings
among the two states in \eqref{def-value}  into account, 
which does {never} appear in the context of 
single Hamilton-Jacobi equations. The key ingredients in this approach
consist of obtaining existence and stability results of extremal curves of \eqref{def-value}.
It is fairly straightforward to prove the existence of extremal curves by using techniques
from calculus of variations. However,
representation formulas \eqref{def-value} are implicit in some sense and
prevent us from deriving a stability result (see Theorem \ref{thm:stability}).
In order to over come this difficulty, we give  more deterministic formulas
for the value functions of \eqref{def-value} by explicit calculations
 in Theorem \ref{thm:representation}. By using the new formulas,
which are {more} intuitive, we are able to derive Theorem \ref{thm:stability}, and
hence large time behavior results.
We here just focus on the case where the coupling coefficients of (C) are constant for the sake of clarity.
It is straightforward to check that our approach works well for the general cases of variable coefficients,
i.e. $c_i \in C(\T^n,(0,\infty))$ for $i=1,2$.

Let us call attention to the forthcoming paper \cite{CGMT1} by Cagnetti, Gomes and the authors,
which provides a completely new and unified approach to the study of large time behaviors of 
both single and weakly coupled systems of Hamilton--Jacobi equations.
A new and different proof of Theorem \ref{thm:main} is derived as well.

After this paper was finished, we learnt that Nguyen \cite{VN} also achieved some similar results independently by using the PDE approach introduced by Barles and Souganidis \cite{BS}. 
We also refer to the interesting recent paper by
Davini and Zavidovique \cite{DZ} on the study of Aubry sets for weakly coupled systems.
\smallskip

The paper is organized as follows. In Section \ref{rep-form} 
we establish new representation formulas,
which are more explicit and useful for our study here. 
We then derive the existence of extremal curves in Section \ref{extrem-exist},
which is pretty standard in the theory of optimal control and calculus of variations.
Section \ref{extrem-stabi} concerns the study of stability of extremal curves.
This section plays the key roles in this paper and allows us to
overcome the technical difficulties coming from the coupling terms. 
See Remarks \ref{difficulty-1} and \ref{difficulty-2} for details. 
Section \ref{proof-main} is devoted to the proof of Theorem \ref{thm:main}. 
We derive generalization results for systems of $m$-equations for $m \ge 2$
in Remark \ref{generalization}.
Finally, some lemmata concerning verifications of optimal control formulas for (C) in Section \ref{rep-form} are recorded in Appendix for readers' convenience.
\smallskip


\section{Preliminaries} \label{rep-form}

In this section, we establish new representation formulas, which 
give us a clearer intuition about the switching states of the systems.
The new formulas allow us to perform deep studies 
on the extremal curves in Sections \ref{extrem-exist}, and \ref{extrem-stabi}.  
For every interval $I \subset \R$ and subset $S \subset \R^m$ for $m\in\N$, we denote by
$\AC(I,S)$ the set of all absolutely continuous functions $\gam:I \to S$. We write
$\AC(I)$ to denote $\AC(I,\T^n)$ for simplicity.

\begin{lem}\label{lem:half1}
Let $\nu$ be a Markov process defined by \eqref{markov}
with $c_1=c_2=1$ and $\nu(0)=i$ for $i\in\{1,2\}$ and 
set $p_{j}(t):=\bP(\nu(t)=j)$ for $j\in\{1,2\}$. 
Then we have 
\[
p_j(t)=1/2+e^{2t}(p_j(0)-1/2) \ \text{for all} \ t<0. 
\]
In particular, $p_{j}(t)\to 1/2$ as $t\to-\infty$ for any $j\in\{1,2\}$. 
\end{lem}

\begin{proof}
By the definition of \eqref{markov} we have for $t<0$ and $s>0$ small enough
\begin{align*}
&p_{j}(t-s)\\
=&\,
\bP(\nu(t-s)=j\mid \nu(t)=i)\bP(\nu(t)=i)
+
\bP(\nu(t-s)=j\mid \nu(t)=j)\bP(\nu(t)=j)\\
=&\,
(s+o(s))(1-p_j(t))+(1-s-o(s))p_j(t). 
\end{align*}
Therefore, 
\[\frac{p_j(t-s)-p_j(t)}{s}=(1+\frac{o(s)}{s})(1-2p_j).\]
Sending $s\to0$ yields $\dot{p}_j=2p_j-1$, which 
implies the conclusion, i.e.,   
$p_j(t)=1/2+e^{2t}(p_j(0)-1/2)$ for all $t<0$.  
\end{proof}

A straightforward result of Lemma \ref{lem:half1} is 
\begin{lem}\label{lem:half2}
Let $\phi_i$ be any functions in $C(\T^n)$ for $i=1,2$. 
We have 
\[
\E_{i}[\phi_{\nu(t)}(x)]
=\frac{1}{2}(1+e^{2t})\phi_i(x)
+\frac{1}{2}(1-e^{2t})\phi_j(x)
\]
for all $x\in\T^n$,
$t<0$, and $i=1,2$, where we take $j$ so that 
$\{i,j\}=\{1,2\}$. 
\end{lem}
\begin{rem}
In general 
if $c_1, c_2>0$ are arbitrary constants, then we have
\[
\E_{i}[\phi_{\nu(t)}(x)]
=\frac{1}{c_1+c_2}(c_j+c_ie^{(c_1+c_2)t})\phi_i(x)
+\frac{c_i}{c_1+c_2}(1-e^{(c_1+c_2)t})\phi_j(x)
\]
for all $x\in\T^n$,
$t<0$, and $i=1,2$, where we take $j$ so that 
$\{i,j\}=\{1,2\}$. 
\end{rem}

It turns out that the value function of optimal control problems \eqref{def-value}   can be written in more explicit forms without using continuous Markov chains as follows by using the Fubini theorem.
\begin{thm}\label{thm:representation}
Let $u_i$ be the value functions defined by \eqref{def-value}. 
Then we can write them as 
\begin{multline}\label{C-rep}
u_i(x,t)=\inf \Big \{
\int_{-t}^{0} \dfrac{1}{2}(1+e^{2s}) L_i(\gam(s),\dot \gam(s))\,ds 
+ \dfrac{1}{2}(1+e^{-2t})g_{i}(\gam(-t))\\
+
\int_{-t}^{0} \dfrac{1}{2}(1-e^{2s}) L_j(\gam(s),\dot \gam(s))\,ds 
+ \dfrac{1}{2}(1-e^{-2t})g_{j}(\gam(-t))\
\mid \ \gam \in \AC([-t,0]),\ \gam(0)=x
\Big \}.
\end{multline}
Moreover, $u_i$ are uniformly continuous on $\cQ$ and 
the pair $(u_1,u_2)$ is the unique viscosity solution 
of {\rm (C)}. 
\end{thm}
We call $(1/2)(1+e^{2s})$ and $(1/2)(1-e^{2s})$ for $s<0$ 
{\it the weights} corresponding to {\rm (C)}, which comes from 
the random switchings among the two states in \eqref{def-value}.

\begin{proof}
By Fubini's theorem and Lemma \ref{lem:half2} 
we have  
\begin{align*}
&\E_{i}\Big[\int_{-t}^{0}
L_{\nu(s)}(\gam(s),\dot{\gam}(s))\,ds
+g_{\nu(-t)}(\gam(-t))\Big]\\
=&\, 
\int_{-t}^{0}
\E_{i}\Big[L_{\nu(s)}(\gam(s),\dot{\gam}(s))\Big]
\,ds
+\E_{i}\big[g_{\nu(-t)}(\gam(-t))\big]\\
=&\, 
\int_{-t}^{0} \dfrac{1}{2}(1+e^{2s}) L_i(\gam(s),\dot \gam(s))\,ds 
+ \dfrac{1}{2}(1+e^{-2t})g_{i}(\gam(-t))\\
&+
\int_{-t}^{0} \dfrac{1}{2}(1-e^{2s}) L_j(\gam(s),\dot \gam(s))\,ds 
+ \dfrac{1}{2}(1-e^{-2t})g_{j}(\gam(-t))
\end{align*}
for any $\gam\in\AC([-t,0])$, 
which implies the equality \eqref{C-rep}.

In Appendix we prove that $u_i$   
are uniformly continuous on $\cQ$ and the pair $(u_1,u_2)$ gives a solution 
of (C). 
In the previous paper \cite{MT2}, we showed that the pair $(u_1,u_2)$ defined by \eqref{def-value}
solves (C) already. But we present it in a different way by using the new formula \eqref{C-rep} itself
to make the paper self-contained.
\end{proof}

Let $(v_1,v_2,0)$ be a solution of (E).  
Without loss of generality,
we may assume that the ergodic constant $c=0$ henceforth.
We notice that $v_i$ satisfies 
\begin{multline}\label{E-value}
v_i(x)=\inf  \Big \{ \E_i \Big [
\int_{-t}^0 L_{\nu(s)}(\gam(s), \dot \gam(s))\,ds +v_{\nu(-t)}(\gam(-t)) \Big ] \mid\\
\ \gam\in\AC((-\infty,0]) \ \text{with} \ \gam(0)=x \Big \},
\end{multline}
where $\nu$  is a $\{1,2\}$-valued process which 
is a continuous-time Markov chain
satisfying \eqref{markov} such that $\nu(0)=i$.

\begin{prop}\label{prop:sub2}
Let $(v_1,v_2,0)$ be a subsolution of {\rm (E)}. Then, 
\[
v_i(x)\le 
\E_i\Big[
\int_{-t}^{0}
L_{\nu(s)}(\gam(s),\dot{\gam}(s))\,ds
+v_{\nu(-t)}(\gam(-t))\Big]
\]
for all $t\ge0$, $\gam\in\AC([-t,0])$ with $\gam(0)=x$. 
\end{prop}

\begin{lem}\label{lem:composite} 
Let $t>0$, $v_i\in\W(\T^n)$ for $i=1,2$ and  
$\gam\in\AC([-t,0],\T^n)$ with $\gam(0)=x$.   
We have $v_i\circ \gam\in\AC([-t,0],\R)$ and 
there exists a function $p_i\in \Li((-t,0),\R^n)$ such that 
\begin{align}
&
v_i(x)=
\E_i\Big[
\int_{-t}^{0}
p_{\nu(s)}(s) \cdot\dot{\gam}(s)+
\sum_{j=1}^{2}(v_{\nu(s)}-v_j)(\gam(s))\,ds
+v_{\nu(-t)}(\gam(-t))\Big], \label{ineq:dominated}\\
&
p_i(s)\in\pl_cv_i(\gam(s))  \nonumber
\end{align}
for a.e. $s\in(-t,0)$ and $i\in\{1,2\}$. 
Here $\pl_cv_i$ denotes the Clarke differential 
of $v_i$ which is defined as
\[
\pl_cv_i(x)=\bigcap_{r>0}\,\ol{\rm co}\,
\{Dv_i(y)\mid y\in B(x,r),\ v_i \textrm{ is differentiable at }y\}
\  \textrm{ for }x\in\T^n,
\]
where $B(x,r):=\{y\in\R^n\mid |x-y|<r\}$,
and for $A\subset \R^n$, $\ol{\rm co}\,A$ denotes the closed convex hull of $A$.
\end{lem}

\begin{proof}
Fix any $i\in\{1,2\}$. 
Let $\rho\in C^{\infty}(\R^n)$ be a standard mollification kernel, 
i.e., $\rho\ge0$, $\supp\rho\subset B(0,1)$ and $\int_{\R^n}\rho(x)\,dx=1$. 
Set $\rho^{k}(x):=k^{n}\rho(kx)$
for $k \in \N$, and 
\begin{align*}
\psi^{k}(i,t)=\psi_i^k(t):=(\rho^{k}\ast v_i)(\gam(t)) \ \textrm{and} \ 
p^k(i,t)=p_i^k(t):=D(\rho^{k}\ast v_i)(\gam(t))  
\end{align*}
for all $t\in(0,T)$. 
By the {It\^o} formula for a jump process we have 
\begin{align*}
&
\E_{i}\Big[\psi^k(\nu(0),0)-\psi^k(\nu(-t),-t)\Big]\\
=&\, 
\E_{i}\Big[
\int_{-t}^{0}p^{k}(\nu(s),s)\cdot\dot{\gam}(s)\,ds
+\int_{-t}^{0}\sum_{j=1}^{2}
\big(\psi^{k}(j,s)-\psi^{k}(\nu(s),s)\big)\,ds
\Big]. 
\end{align*}

Note that $\psi_i^k\to v_i(\gam(\cdot))$ uniformly on $[0,t]$ as $k\to\infty$ 
and moreover passing to a subsequence if necessary, 
we may assume that for some 
$p_i\in\Li((0,T),\R^n)$, 
$p_i^{k} \rightharpoonup p_i$ 
weakly star in $L^\infty((-t,0))$ as $k\to\infty$,  
which implies \eqref{ineq:dominated}.

It remains to show that $p_i(s)\in\pl_cv_i(\gam(s))$ for a.e. $s\in(-t,0)$.  
Since $\{p_i^k\}_{k\in\N}$ is weakly convergent to $p_i$ in $L^2((-t,0),\R^n)$, 
by the Mazur theorem, there is a sequence $\{q_i^k\}_{k\in\N}\subset L^\infty((-t,0),\R^n)$ 
such that 
\begin{equation}\label{convex-hull}
q_i^k\to p_i \ \textrm{ strongly in }L^2((-t,0),\R^n)\textrm{ as }k\to\infty, \ 
q_i^k\in {\rm co}\,\{p_i^j\mid j\geq k\}
\end{equation}
for all $j\in\N$.
We may thus assume by its subsequence if necessary that 
\[
q_i^k(s)\to p_i(s) \  \hbox{ for } a.e.  \ s\in (-t,0)\ \textrm{ as } \ k\to\infty. 
\]

Now, noting that 
$D(\rho_k\ast v)(x)=\int_{y\in B(x,1/k)}\rho_k(x-y)Dv_i(y)\,dy$ 
for any $x\in\T^n$ and $k\in\N$, 
we find that 
\[
p_i^k(s)\in \ol{\co}\{Dv_i(y) \mid y\in B(\gam(s),1/k), 
\ v_i \ \textrm{is differentiable at} \ y\}
\]
for any {$s\in(-t,0)$}. 
Therefore, 
\[
q_i^k(s)\in \ol{\co}\{Dv_i(y) \mid y\in B(\gam(s),1/k), 
\ v_i \ \textrm{is differentiable at} \ y\}
\]
for any $s\in(-t,0)$. 
Since $q_i^k(s)\to p_i(s)$ for $a.e.$ $s\in(-t,0)$ as $k\to\infty$, 
we get 
\[
p_i(s)\in \bigcap_{r>0}\,\ol{\co}
\{Dv_i(y)\mid y\in B(\gam(s),r),\ v_i \textrm{ is differentiable at }y\}
=\pl_cv_i(\gam(s))
\]
for $a.e.$ $s\in(-t,0)$. 
\end{proof}

\begin{proof}[Proof of Proposition {\rm\ref{prop:sub2}}]
Let $\gam\in\AC([-t,0])$ with $\gam(0)=x$ and 
$p_i$ be the functions given by Lemma \ref{lem:composite}. 
In view of Lemma \ref{lem:composite} we have 
\begin{align}
v_i(x)&= 
\E_i\Big[
\int_{-t}^{0}
p_{\nu(s)}(s)\cdot\dot{\gam}(s)+
\sum_{j=1}^{2}(v_{\nu(s)}-v_j)(\gam(s))\,ds
+v_{\nu(-t)}(\gam(-t))\Big] \nonumber\\
&\le
\E_i\Big[
\int_{-t}^{0}
H_{\nu(s)}(\gam,p_{\nu(s)})
+L_{\nu(s)}(\gam,\dot{\gam})
+\sum_{j=1}^{2}(v_{\nu(s)}-v_j)(\gam)\,ds
+v_{\nu(-t)}(\gam(-t))\Big] \label{ineq-composite}
\\
&\le
\E_i\Big[
\int_{-t}^{0}
L_{\nu(s)}(\gam,\dot \gam)\,ds
+v_{\nu(-t)}(\gam(-t))\Big]. 
\qedhere 
\end{align}
\end{proof}


\section{Existence of Extremal Curves}\label{extrem-exist}
Let $(v_1,v_2,0)$ be a solution of (E). 
For any interval 
$[a,b] \subset (-\infty,0]$, we denote by $\cE([a,b],x,i,(v_1,v_2))$ the set of all curves
$\gam \in \AC([a,b])$, 
which will be called an \textit{extremal curve} on $[a,b]$ 
such that $\gam(b)=x$ and for any $[c,d] \subset [a,b]$,
$$
\E_i [ v_{\nu(d)}(\gam(d)) ]=\E_i \Big[
\int_{c}^d L_{\nu(s)}(\gam(s), \dot \gam(s))\,ds +v_{\nu(c)}(\gam(c)) \Big] 
$$
with a continuous-time Markov chain $\nu$ such that $\nu(0)=i$ 
and satisfies \eqref{markov}.

\begin{thm}\label{thm:extremal-c}
Let $(v_1,v_2,0)$ be a solution of {\rm (E)}. 
Then $\cE((-\infty,0],x,i,(v_1,v_2)) \ne \emptyset$.
\end{thm}

In order to avoid technical difficulties we make 
the following additional assumptions in this section
which are not necessary to get Theorem \ref{thm:extremal-c} and
Theorem \ref{thm:main}. 
We refer the readers to \cite[Section 6]{I2008} for the detail of general 
settings.

\begin{itemize}
\item[(A3)] $H_i \in C^2(\T^n \times \R^n)$ and there exists $\theta>0$ such that $D^2_{pp}H_i \ge \theta I$ 
for $i=1,2$, where $I$ is the unit matrix of size $n$.

\item[(A4)] There exists a constant $C>0$ such that 
$$
\dfrac{1}{2C}|p|^2-C \le H_i(x,p) \le \dfrac{C}{2}(|p|^2+1) \ 
\text{for} \ x \in \T^n,\ p \in \R^n, \ i=1,2.
$$
\end{itemize}
Note that in this case we can easily see that 
$L_i\in C^{2}(\T^n\times\R^n)$ are uniformly convex and satisfy
\begin{equation}\label{Lagran}
 \dfrac{1}{2C}|q|^2-C \le L_i(x,q) \le \dfrac{C}{2}(|q|^2+1) \ 
\text{for} \ x \in \T^n,\ q \in \R^n, \ i=1,2.
\end{equation}

\begin{lem} \label{extrem-lem2}
Let $(v_1,v_2,0)$ be a solution of {\rm (E)}. 
Then $\cE([-1,0],x,i,(v_1,v_2)) \ne \emptyset$.
\end{lem}

\begin{proof}
By \eqref{E-value} there exists a sequence of curves
$\{\gam_k\} \subset \AC([-1,0])$ with $\gam_k(0)=x$ such that
$$
v_i(x)+\dfrac{1}{k} >   \E_i \Big [
\int_{-1}^0 L_{\nu(s)}(\gam_k(s), \dot \gam_k(s))\,ds +v_{\nu(-1)}(\gam_k(-1)) \Big ].
$$
Since $v_i$ are bounded, we have 
\begin{equation}\label{bd1}
\E_i \Big [
\int_{-1}^0 L_{\nu(s)}(\gam_k(s), \dot \gam_k(s))\,ds\Big ]\le C \ 
\text{for some} \ C>0.  
\end{equation}
Combining \eqref{bd1} and \eqref{Lagran}, 
we deduce that $\|\dot \gam_k\|_{L^2(-1,0)} \le M$ for some $M>0$.
For any $-1 \le a<b \le 0$, we have
$$
|\gam_k(b)-\gam_k(a)| \le \int_a^b|\dot \gam_k(s)|\,ds
\le \Big [  \int_a^b|\dot \gam_k(s)|^2\,ds \Big ]^{1/2}
\Big [ \int_a^b 1\,ds \Big ]^{1/2} \le M |b-a|^{1/2}.
$$
By the Arzela--Ascoli theorem and the weak compactness, 
by passing to a subsequence if necessary, 
$\{\gam_k\}$ converges to 
$\gam \in \AC([-1,0])$ uniformly,
and $\{\dot \gam_k\}$ converges weakly to $\dot \gam$ in $L^2(-1,0)$.

Now we prove that
\begin{equation}\label{lower}
\E_i \Big [
\int_{-1}^0 L_{\nu(s)}(\gam(s), \dot \gam(s))\,ds\Big ]
\le \liminf_{k\to \infty}\E_i \Big [
\int_{-1}^0 L_{\nu(s)}(\gam_k(s), \dot \gam_k(s))\,ds\Big ].
\end{equation}
This is a standard part in the theory of calculus of variations 
but let us present it here for the sake of clarity.
The convexity of $L_i$ gives us that
\begin{align*}
&L_i(\gam_k(s),\dot \gam_k(s)) \ge L_i(\gam_k(s),\dot \gam(s)) +
 D_q L_i(\gam_k(s),\dot \gam(s)) \cdot (\dot \gam_k(s) - \dot \gam(s))\\
 =&\  L_i(\gam_k(s),\dot \gam(s)) 
 +
[D_q L_i(\gam_k(s),\dot \gam(s))- D_q L_i(\gam(s),\dot \gam(s))]\cdot (\dot \gam_k(s) - \dot \gam(s))\\
&
+ D_q L_i(\gam(s),\dot \gam(s)) \cdot (\dot \gam_k(s) - \dot \gam(s)).
\end{align*}
Since $\gam_k$ converges uniformly to $\gam$, we employ 
the Lebesgue dominated convergence theorem to get that
\begin{equation}\label{lower-1}
\lim_{k\to \infty}\E_i \Big [
\int_{-1}^0 L_{\nu(s)}(\gam_k(s), \dot \gam(s))\,ds\Big ]
= \E_i \Big [
\int_{-1}^0 L_{\nu(s)}(\gam(s), \dot \gam(s))\,ds\Big ].
\end{equation}
We use \eqref{Lagran} again to yield that
\[
|D_q L_i(x,q)| \le C(|q|+1) \ 
\text{for}\ x\in \T^n,\ q \in \R^n, \ i=1,2.
\]
It it then straightforward by using the above and 
the Lebesgue dominated convergence theorem 
to see that
\begin{equation}\label{lower-2}
\lim_{k \to \infty} \E_i \Big [ 
\int_{-1}^0 ( D_q L_{\nu(s)}(\gam_k(s),\dot \gam(s))- D_q L_{\nu(s)}(\gam(s),\dot \gam(s)))
\cdot (\dot \gam_k(s) - \dot \gam(s)) \,ds \Big ]=0
\end{equation}
Besides, the weak convergence of $\{\dot \gam_k\}$ to $\dot \gam$ in $L^2(-1,0)$ implies
\begin{equation*}
\lim_{k \to \infty} \E_i \Big [ \int_{-1}^0
 D_q L_{\nu(s)}(\gam(s),\dot \gam(s)) \cdot (\dot \gam_k(s) - \dot \gam(s))
\,ds \Big]=0.
\end{equation*}
We combine \eqref{lower-1}, \eqref{lower-2}, and the above 
to get \eqref{lower}.
Thus, $\gam$ satisfies
\begin{equation}\label{gam-1}
v_i(x) \ge  \E_i \Big [
\int_{-1}^0 L_{\nu(s)}(\gam(s), \dot \gam(s))\,ds +v_{\nu(-1)}(\gam(-1)) \Big ].
\end{equation}
On the other hand, for any $-1 \le a<b \le 0$,
\begin{align*}
&v_i(x) \le  \E_i \Big [
\int_{b}^0 L_{\nu(s)}(\gam(s), \dot \gam(s))\,ds +v_{\nu(b)}(\gam(b)) \Big ],\\
&\E_i[v_{\nu(b)}(\gam(b)) ] \le  \E_i \Big [
\int_{a}^b L_{\nu(s)}(\gam(s), \dot \gam(s))\,ds +v_{\nu(a)}(\gam(a)) \Big ],\\
&\E_i[v_{\nu(a)}(\gam(a)) ] \le  \E_i \Big [
\int_{-1}^a L_{\nu(s)}(\gam(s), \dot \gam(s))\,ds +v_{\nu(-1)}(\gam(-1)) \Big ].
\end{align*}
The above inequalities together with \eqref{gam-1} yield
the conclusion that $\gam \in \cE([-1,0],x,i,(v_1,v_2))$.
\end{proof}

\begin{proof}[Proof of Theorem {\rm \ref{thm:extremal-c}}]
Fix $x\in\T^n$ and $i\in\{1,2\}$. 
We define the sequence 
$\{\gam^k\}_{k\in\N}\subset\AC([-k,-k+1])$ 
recursively as  $\gam^k\in\cE([-k,-k+1],x_{k-1},i,(v_1,v_2))$, 
where 
$x_k:=\gam^k(-k)$ and $x_0=x$. 
Define the curve $\gam \in \AC((-\infty,0])$ by 
$\gam(s)=\gam^k(s)$ for $s \in [-k,-k+1]$ for $k \in \N$. 
Then it is clear to see that $\gam \in  \cE((-\infty,0],x,i,(v_1,v_2))$. 
\end{proof}


\section{Stability on the Extremal Curves}\label{extrem-stabi}
In this section, we establish the following stability result, which 
plays a key role in the proof of Theorem \ref{thm:main}. 
\begin{thm}[Scaling Result]\label{thm:stability}
Let $(v_1,v_2,0)$ be a solution of {\rm (E)}. 
For any $\tau,T\in(0,\infty)$ with 
$\tau<T$ such that 
$\tau/(T-\tau)<\del_0$, 
where $\del_0$ appears in Lemma {\rm \ref{stab-prop2}}, 
and $\gam\in\cE((-\infty,0],x,i,(v_1,v_2))$, we have 
\begin{align}
&u_i(x,T)-\E_i[u_{\nu(-T)}(\gam(-T),\tau)] \nonumber\\
\le&\,  
v_i(x)-\E_i[v_{\nu(-T)}(\gam(-T))] + 
(1+\dfrac{\tau T}{T-\tau})\om(\dfrac{\tau}{T-\tau})
\label{stab-1}
\end{align}
for a fuction $\om:[0,\infty) \to [0,\infty)$ 
which is continuous and $\om(0)=0$. 
\end{thm}


\begin{lem}\label{stab-prop1}
For any $T>0$ and $\gam \in  \cE((-\infty,0],x,i,(v_1,v_2))$.  
There exists $(p_1,p_2) \in L^\infty((-T,0),\R^n)^2$ such that
\begin{align*}
&L_i(\gam(t),\dot \gam(t))+H_i(\gam(t),p_i(t))=p_i(t) \cdot \dot \gam(t), \\   
&H_i(\gam(t),p_i(t))+v_i(\gam(t))-v_j(\gam(t)) = 0, \ \text{and} \ 
p_i(t) \in \partial_c v_i(\gam(t))  
\end{align*}
for a.e. $t \in (-T,0)$.
\end{lem}

\begin{proof}
By Lemma \ref{lem:composite} there exists 
$(p_1,p_2) \in L^\infty((-T,0),\R^n)^2$ such that 
$p_i(t) \in \partial_c v_i(\gam(t))$ for $a.e.$ $t\in (-T,0)$ 
and satisfies \eqref{ineq-composite} 
in the proof of Proposition \ref{prop:sub2}.
Also, note that 
by the convexity of $H_i$ and the definition of 
$L_i$,  
we have 
$H_i(\gam(t),p_i(t))+v_i(\gam(t))-v_j(\gam(t)) \le 0$ and 
$H_i(\gam(t),p_i(t))+L(\gam(t),\dot\gam(t))\ge p_i(t)\cdot\dot\gam(t)$ 
for $a.e.$ $t \in (-T,0)$ and $i=1,2$.  
Since $\gam$ is an extremal curve,
all inequalities above must become the equalities, which 
give the desired conclusion. 
\end{proof}

\begin{lem}\label{stab-prop2}
Let $(v_1,v_2,0)$ be a solution of {\rm(E)}. 
There exists $\del_0>0$ such that 
for any $\ep\in[0,\del_0]$ and $\gam \in \cE((-\infty,0],x,i,(v_1,v_2))$ 
we have
$$
L_i(\gam(t),(1+\ep)\dot \gam(t)) \le
(1+\ep) L_i(\gam(t),\dot \gam(t))-\ep (v_i-v_j)(\gam(t))+\ep \om(\ep)
$$
for a fuction $\om:[0,\infty) \to [0,\infty)$ 
which is continuous and $\om(0)=0$.
\end{lem}

\begin{proof}
Let $(p_1,p_2)$ be the pair of functions given by Lemma \ref{stab-prop1}. 
We notice that
$$
H_i(\gam(t),p_i(t))+H_j(\gam(t),p_j(t))=0
\ \text{for a.e.} \ t \in (-\infty,0]
$$
by Lemma \ref{stab-prop1}. 
Set 
\begin{align*}
&Q:=\{(x,p_1,p_2)\in \T^n \times \R^{2n}| \ H_1(x,p_1)+H_2(x,p_2) = 0\}, \\
&S:=\{(x,q_1,q_2)|\ q_i \in D_p^{-} H_i(x,p_i) \
\text{for some} \ (x,p_1,p_2)\in Q \} 
\end{align*}
and then $Q$ and $S$ are compact
in $\T^n \times \R^{2n}$ in view of the coercivity of $H_i$. 
We notice that $(\gam(t),\dot \gam(t),\dot \gam(t)) \in S$ for a.e. $t \in (-\infty,0)$ and thus $|\dot{\gam}(t)|\le M$ for some $M>0$. 
We choose $\del_0\in(0,1)$ so that
$(x,(1+\ep)\dot{\gam})\in\inter(\dom L_1\cap\dom L_2)$ 
for all $\ep\in[0,\del_0]$, 
where  $\dom L_i:=\{(x,\xi)\in\T^n\times\R^n\mid L_i(x,\xi)<\infty\}$.

By Lemma \ref{stab-prop1},
\begin{align}
&
L_i(\gam(t),\dot \gam(t))=p_i(t)\cdot \dot \gam(t)-H_i(\gam(t),p_i(t)) \nonumber\\
=&\,
D_q L_i(\gam(t),\dot \gam(t)) \cdot \dot \gam(t)+ (v_i-v_j)(\gam(t)).
\label{add-2}
\end{align}
Note that since $H_i(x,\cdot)$ are strictly convex,
$D_qL_i(x,\xi)$ exists, and is continuous on $\dom L_i$.

Due to the mean value theorem and \eqref{add-2},  
there exists $\theta_t\in(0,1)$ and a fuction $\om:[0,\infty) \to [0,\infty)$ 
which is continuous and $\om(0)=0$ such that 
\begin{align*}
L_i(\gam(t),(1+\ep)\dot \gam(t))
=\ & L_i(\gam(t),\dot \gam(t))+\ep D_q L_i(\gam(t),(1+\theta_t \ep)\dot \gam(t))\cdot \dot \gam(t)\\
\le \ &  L_i(\gam(t),\dot \gam(t))+\ep D_q L_i(\gam(t),\dot \gam(t))\cdot \dot \gam(t)
+ \ep |\dot \gam(t)| \om(\ep |\dot \gam(t)|) \\
\le \ & (1+\ep) L_i (\gam(t),\dot \gam(t)) - \ep (v_i-v_j)(\gam(t))
+ \ep \tilde{\om}(\ep),   
\end{align*}
where we set $\tilde{\om}(r):=M\max_{s \in [0,Mr]} \om(s)$. 
\end{proof}

\begin{rem} \label{difficulty-1}
We notice that the result of Lemma \ref{stab-prop2} is different from the similar one
for single equations (see \cite[Lemma 7.2]{I2008} for details).
More precisely, the natural appearance of the coupling terms $-\ep (v_i-v_j)(\gam(t))$
makes the analysis for weakly coupled systems more difficult. 
We could not proceed to establish large time behavior results in a crude way.
It turns out that {the weights} $(1/2)(1+e^{2t})$ and $(1/2)(1-e^{2t})$ 
for $t<0$ are the key factors helping us overcome this difficulty as in the proof of Theorem
\ref{thm:stability} below.
\end{rem}

\begin{proof}[Proof of Theorem {\rm \ref{thm:stability}}]
Set $\ep:=\tau/(T-\tau)$ and $T_\ep:=T/(1+\ep)$.
Notice that $T=T_\ep+\ep T_\ep=T_\ep+\tau$.
We have
\begin{align*}
&u_i(x,T)= u_i(\gam(0),T) = u_i(\gam(0),T_\ep+\tau)\\
&=\inf \Big \{ 
\E_i \Big [ 
\int_{-T_\ep}^{0} L_{\nu(s)}(\eta(s),\dot \eta(s))\,ds + u_{\nu(-T_\ep)}(\eta(-T_\ep),\tau)
\Big ]\ | \ \eta\in \AC([-T_\ep,0])\ \text{with}\ \eta(0)=x
\Big \}.
\end{align*}
Take $\gam \in \cE((-\infty,0],x,i,(v_1,v_2))$ and 
set $\eta(s):=\gam((1+\ep)s)$ to derive that 
$$
u_i(x,T) \le 
\E_i \Big [ 
\int_{-T_\ep}^{0} L_{\nu(s)}(\gam((1+\ep)s),(1+\ep)\dot \gam((1+\ep)s))\,ds + 
u_{\nu(-T_\ep)}(\gam(-T),\tau)
\Big ]. 
$$
Make the change of variable $t=(1+\ep)s$ and use 
Lemma \ref{lem:half2} to get
\begin{align}
u_i(x,T) \le &\,
\E_i \Big [ 
\int_{-T}^{0}\dfrac{1}{1+\ep} L_{\nu(t/(1+\ep))}(\gam(t),(1+\ep)\dot \gam(t))\,dt + 
u_{\nu(-T_\ep)}(\gam(-T),\tau)
\Big ]
\label{rem-easy}\\
=& 
\int_{-T}^{0} 
\dfrac{1+e^{2t/(1+\ep)}}{2(1+\ep)} L_i(\gam(t),(1+\ep)\dot \gam(t))\,dt
+\int_{-T}^{0} 
\dfrac{1-e^{2t/(1+\ep)}}{2(1+\ep)}  L_j(\gam(t),(1+\ep)\dot \gam(t))\,dt\nonumber\\
&+\E_i \big [ u_{\nu(-T_\ep)}(\gam(-T),\tau)
\big ]
. 
\label{key0}
\end{align}
We use Lemma \ref{stab-prop2} in the above inequality to deduce
\begin{align}
&u_i(x,T) - \E_i \big [ u_{\nu(-T_\ep)}(\gam(-T),\tau)\big]
\nonumber\\
 \le& \int_{-T}^{0} \dfrac{1}{2}(1+e^{2t/(1+\ep)})L_i(\gam(t),\dot \gam(t))
+ \dfrac{1}{2}(1-e^{2t/(1+\ep)})L_j(\gam(t),\dot \gam(t)) \,dt
\nonumber\\
&+\dfrac{\ep}{1+\ep} \int_{-T}^{0} e^{2t/(1+\ep)} (v_j-v_i)(\gam(t))\,dt + T\ep \om(\ep).
\label{key1}
\end{align}

We use the fact that $v_j-v_i$ is bounded in $\T^n$ to derive that
\begin{equation}\label{key2}
\Big | \dfrac{\ep}{1+\ep} \int_{-T}^{0} e^{2t/(1+\ep)} (v_j-v_i)(\gam(t))\,dt   \Big |
\le C\ep \int_{-T}^0 e^{2t/(1+\ep)}\,dt \le C\ep.
\end{equation}
Furthermore, for $t<0$, $|e^{2t/(1+\ep)}-e^{2t}| \le 
{-2t\ep e^{2t/(1+\ep)}}$.
This together with the facts that $u_i$ are bounded and $|\dot \gam(t)| \le M$ imply
$\E_i \big[ u_{\nu(-T_\ep)}(\gam(-T),\tau)\big]
\le \E_i \big[ u_{\nu(-T)}(\gam(-T),\tau)\big]
+C \ep$, and
$$
\Big |
 \int_{-T}^{0} (e^{2t/(1+\ep)}-e^{2t})L_k(\gam(t),\dot \gam(t)) \,dt
\Big | \le 
-C_1\ep  \int_{-T}^{0} te^{2t/(1+\ep)}\,dt
\le C_2\ep,
$$
for $k=1,2$ and $C_1,C_2>0$ independent of $\ep$.

Summing up everything, we obtain
\begin{align*}
&u_i(x,T) - \E_i \big[ u_{\nu(-T)}(\gam(-T),\tau)\big]\\
 \le& \int_{-T}^{0}\Big [ \dfrac{1}{2}(1+e^{2t})L_i(\gam(t),\dot \gam(t))
+ \dfrac{1}{2}(1-e^{2t})L_j(\gam(t),\dot \gam(t)) \Big ]\,dt+C\ep +T\ep \om(\ep)\\
=&\, v_i(x) - \E_i \big[ v_{\nu(-T)}(\gam(-T))\big]+C \dfrac{\tau}{T-\tau}+\dfrac{\tau T}{T-\tau} \om (\dfrac{\tau}{T-\tau}), 
\end{align*}
which is the desired conclusion. 
\end{proof}

\begin{rem}\label{difficulty-2}
The new representation formula \eqref{C-rep}
with the weights $(1/2)(1+e^{2t})$ and $(1/2)(1-e^{2t})$
for $t<0$ appears naturally in both the statement and the proof of Theorem \ref{thm:stability}
pointing out a major difference between single equations and weakly coupled systems.
With new representation formula \eqref{C-rep}, we could explicitly calculate \eqref{key0}
and \eqref{key1} and thus identify the main obstacle coming from the coupling term,
the second last term in \eqref{key1}.
As mentioned in Remark \ref{difficulty-1}, we could not estimate the coupling term in a crude way. 
For instance, in \eqref{rem-easy} we can easily see by Lemma \ref{stab-prop2} that
\begin{align*}
&\E_i \Big [ 
\int_{-T}^{0}\dfrac{1}{1+\ep} L_{\nu(t/(1+\ep))}(\gam(t),(1+\ep)\dot \gam(t))\,dt 
\Big ]\\
&\le\, 
\E_i \Big [ 
\int_{-T}^{0}L_{\nu(t/(1+\ep))}(\gam(t),\dot \gam(t)) 
-\frac{\ep}{1+\ep}(v_{\nu(t/(1+\ep))}-v_{3-\nu(t/(1+\ep))})\,dt
\Big ]+ \ep T \om(\ep)\\
&\le\, 
\E_i \Big [ 
\int_{-T}^{0}L_{\nu(t/(1+\ep))}(\gam(t),\dot \gam(t))\,dt\Big ]
+\ep T \om(\ep)+CT\ep/(1+\ep) 
\end{align*}
by using the fact that $\|v_1-v_2\|_{L^\infty(\T^n)} \le C$. 
But the last term in the above,
which is of order $O(\tau)$ and does not vanish as $\ep \to 0$,  is not enough to get the large-time asymptotics 
as we can see in the proof of Theorem \ref{thm:main}. 
It turns out that {the weights} played an essential role here
and helped us in establishing {the key estimate \eqref{key2} leading to the large time behavior result.}
\end{rem}


\section{The Proof of Convergence} \label{proof-main}
We define the functions $\ol{u}_i$ and $\ul{u}_i$ 
($i=1,2$) by 
\begin{align*}
\ol{u}_i(x)=\limsup_{s \to \infty} u_i(x,s), \qquad \ul{u}_i(x)=\liminf_{s \to \infty} u_i(x,s).
\end{align*}
By stability of viscosity solutions, we have that $(\ol{u}_1,\ol{u}_2)$ is a subsolution of (E) 
and $(\ul{u}_1,\ul{u}_2)$ is a supersolution of (E). In order to establish large time behavior result,
we need to obtain that $(\ol{u}_1,\ol{u}_2)=(\ul{u}_1,\ul{u}_2)$.
\begin{proof}[Proof of Theorem {\rm\ref{thm:main}}]
Assume by contrary that $(\ol{u}_1, \ol{u}_2) \neq (\ul{u}_1, \ul{u}_2)$.
Take $(\phi_1,\phi_2)$ to be the maximal solution of (E) such that $\phi_i \le \ul{u}_i$ for $i=1,2$.
Without loss of generality, we may assume that there exists $x\in \T^n$ such that
\begin{equation}\label{pm1}
\ol{u}_1(x)+\ul{u}_1(x)-2\phi_1(x)=\max_{i=1,2} \max_{z\in\T^n} \left ( \ol{u}_i(z)+\ul{u}_i(z)-2\phi_i(z) \right)=:\al>0.
\end{equation}

We assume first that $\ol{u}_1(x)>\ul{u}_1(x)$. Take $\gam \in \cE((-\infty,0],x,1,(\phi_1,\phi_2))$. We can choose a sequence
$\{T_m\} \subset (0,\infty)$ converging to $\infty$ such that 
$\lim_{m \to \infty} u_1(x,T_m)=\ol{u}_1(x) > \ul{u}_1(x)$.
Without loss of generality, we assume further that $\gam(-T_m) \to y \in \T^n$ as $m\to \infty$.
We apply Theorem \ref{thm:stability}
 to get
\begin{align*}
&u_1(x,T_m) - 
\Big\{\dfrac{1}{2}(1+e^{-2T_m}) u_1(\gam(-T_m),\tau)
+\dfrac{1}{2}(1-e^{-2T_m}) u_2(\gam(-T_m),\tau)\Big\}\\
\le &\, 
\phi_1(x) - 
\Big\{\dfrac{1}{2}(1+e^{-2T_m}) \phi_1(\gam(-T_m))
+\dfrac{1}{2}(1-e^{-2T_m})\phi_2(\gam(-T_m))\Big\}
+(1+\dfrac{\tau T_m}{T_m-\tau})\om(\dfrac{\tau}{T_m-\tau})
\notag
\end{align*}
for any fixed $\tau>0$ and $m$ large enough.
Let $m \to \infty$ in the above inequality to yield
\begin{equation}\label{pm2}
\ol{u}_1(x) - \dfrac{1}{2}(u_1(y,\tau)+u_2(y,\tau)) \le \phi_1(x) - \dfrac{1}{2}(\phi_1(y)+\phi_2(y)).
\end{equation}
Take $i=\{1,2\}$ such that  $\ol{u}_i(y)-\phi_i(y) \ge \ol{u}_j(y)-\phi_j(y)$ for $j=3-i$.
Choose $\tau=t_n$ where $t_n \to \infty$ such that $u_i(y,t_n) \to \ul{u}_i(y)$ in \eqref{pm2} to get
\begin{equation*}
\ol{u}_1(x) - \phi_1(x) \le \dfrac{1}{2} (\ul{u}_i(y)+\ol{u}_j(y))-\dfrac{1}{2}(\phi_1(y)+\phi_2(y))
\le \dfrac{1}{2}(\ul{u}_i(y)+\ol{u}_i(y))-\phi_i(y),
\end{equation*}
which contradicts \eqref{pm1} as
$$
\ol{u}_1(x)+\ul{u}_1(x)-2\phi_1(x)< 2(\ol{u}_1(x) - \phi_1(x)) \le \ul{u}_i(y)+\ol{u}_i(y)-2\phi_i(y), 
$$
because we are assuming $\ol{u}_1(x)>\ul{u}_1(x)$.

Finally, we need to handle the case where $\ul{u}_1(x)=\ol{u}_1(x)$. It is then immediate 
that $\ul{u}_2(x)=\ol{u}_2(x)$. We will show that this could not happen because
of the maximality of $(\phi_1,\phi_2)$. Define, for $i=1,2$ and $y \in \T^n$,
\begin{multline*}
\psi_i(y)=\inf\Big\{ v_i(y) \mid (v_1,v_2) \ \text{is a supersolution of (E) with} \\
 v_j(x)=\ul{u}_j(x),\ \text{and}\  v_j \ge \phi_j \ \text{for} \ j=1,2 \Big\}.
\end{multline*}
It is clear that $(\psi_1,\psi_2)$ is a supersolution of (E) and furthermore it is a solution of (E) in 
$\T^n \setminus \{x\}$ by Perron's method. We now claim that $(\psi_1,\psi_2)$
is actually a solution of (E), which contradicts the maximality of $(\phi_1,\phi_2)$ and \eqref{pm1}.
For any $p \in D^{-} \psi_1(x)$, it is straightforward that $p \in D^{-}\ol{u}_1(x)$, and
$$
H_1(x,p)+\psi_1(x)-\psi_2(x)=H_1(x,p)+\ol{u}_1(x)-\ol{u}_2(x) =0,
$$
which in view of the characterization of viscosity solution of convex first order equations by
Barron and Jensen \cite{BJ},
yields that $(\psi_1,\psi_2)$ is a solution of (E).
The proof is complete.
\end{proof}

\begin{rem}\label{generalization}
Let us notice that systems of $m$-equations for $m \ge 2$ can be treated in the same way as above.
More precisely, the weakly coupled system of $m$-equations ($m\ge 2$) is consider to be of the form
\begin{equation}\label{general-system}
(u_i)_t + H_i(x,Du_i)+\sum_{j=1}^m c_{ij} u_j=0 \ \text{in} \ \Q,\ \text{for} \ 1 \le i \le m, 
\end{equation}
where $H_i$ satisfy (A1), (A2) for all $i=1,\ldots,m$ and 
\begin{equation*}\label{gen-cij}
c_{ii} \ge 0, \ c_{ij} \le 0 \ \text{for}\ i \ne j,\ \text{and} 
\sum_{i=1}^m c_{ij}=\sum_{j=1}^m c_{ij}=0,
\end{equation*}
then the result of Theorem \ref{thm:main} holds.
For simplicity, we assume further that the matrix $(c_{ij})$ is irreducible, i.e.,
$$
\text{(M) For any} \
I \varsubsetneq \{1,\ldots,m\}, \
\text{there exist} \ i \in I, \ \text{and} \
j \in \{1,\ldots,m\} \setminus I \ \text{such that} \
c_{ij} \ne 0.
$$
Condition (M) is not needed in general and can be removed as in \cite[Section 3.3]{MT1}.

In this general setting, the weights $1/2(1+e^{2s})$ and $1/2(1-e^{2s})$ for $s<0$
are replaced by the general weights $\phi_1,\ldots, \phi_m$, which solve the following system of ODE
$$
\begin{cases}
-(\phi_k)_t + \sum_{j=1}^m c_{kj} \phi_j=0 \ \text{in}\ (-\infty,0),\ \text{for}\ k=1,\ldots,m, \\ 
\phi_k(0)=\del_i^k
\end{cases}
$$
where $\del_i^k=1$ if $k=i$, and $\del_i^k=0$ otherwise for given $i \in \{1,\ldots,m\}$.
It is straightforward to derive that $0 \le \phi_k \le 1$ for all $k$ and
$$
\sum_{k=1}^m \phi_k(s)=1, \ \text{for} \ s \le 0,
$$
and $\lim_{s\to -\infty} \phi_k(s)=1/m$ for $k=1,\ldots,m$.
On the other hand, the matrix $(c_{ij})_{i,j=1}^m$
has a simple eigenvalue $0$, and its other eigenvalues
$\lam_1,\ldots,\lam_{m-1}$ have
positive real parts (see \cite{CLLN} for details).
Hence $\phi_k$ can be written as
\begin{equation*} \label{cij-1}
\phi_k(s) = \frac{1}{m} + \sum_{l=1}^{m-1} a_{kl} e^{\lam_l s}
\end{equation*}
for some constants $a_{kl} \in \C$ for $1 \le k\le m,\  1 \le l \le m-1$.
We then use the fact that $\text{Re} \lam_l >0$ to get
\begin{equation} \label{cij-2}
\int_{-\infty}^0 |\phi_i(s)-\phi_j(s)|\,ds \le C, \ \text{for}\  1 \le i,j \le m.
\end{equation}
Let us emphasize that \eqref{cij-2} is the key point here.

Next, we can also obtain
\begin{multline}\notag
u_i(x,t)=\inf \Big \{ \int_{-t}^0 \sum_{k=1}^m \phi_k(s) L_k(\gam(s),\dot \gam(s))\,ds \\
+\sum_{k=1}^m \phi_k(-t) g_k(\gam(-t))|\
\gam \in \AC([-t,0]),\ \gam(0)=x \Big \}
\end{multline}
by a similar way for the solution of \eqref{general-system} with 
the initial value $u_i(\cdot,0)=g_i$ on $\T^n$ for any $g_i\in C(\T^n)$. 
In order to prove Theorem \ref{thm:main},
we only need to verify the following key estimate concerning the coupling terms, which is similar to \eqref{key2},
\begin{equation*}\label{cij-3}
\big | \int_{-\infty}^ 0 \sum_{k=1}^m \phi_k(t) \sum_{j=1}^m c_{kj} v_j(\gam(t))\,dt \big | \le C.
\end{equation*}
One can see that the above follows directly from \eqref{cij-2} as
\begin{align*}
&\big | \int_{-\infty}^ 0 \sum_{k=1}^m \phi_k(t) \sum_{j=1}^m c_{kj} v_j(\gam(t))\,dt \big |\\
=\, & \big | \int_{-\infty}^ 0 \sum_{k=1}^m \phi_k(t) \sum_{j=1}^m c_{kj} v_j(\gam(t))\,dt-
\int_{-\infty}^ 0 \sum_{k=1}^m \phi_i(t) \sum_{j=1}^m c_{kj} v_j(\gam(t))\,dt \big | \\
\le \,& C \int_{-\infty}^0 \sum_{k=1}^m |\phi_i(t)-\phi_k(t)|\,dt \le C.
\end{align*}
\end{rem}


\section{Appendix}

Let $u_i$ be the value function associated with 
\eqref{def-value}, or equivalently the function defined by 
the right hand side of \eqref{C-rep}.

\begin{prop}[Dynamic Programming Principle]\label{prop:DPP}
For any $x\in \R^n$, $0\le h \le t$ and $i=1,2$, we have
\begin{multline}\label{C-DPP}
u_i(x,t)=\inf \Big \{
\int_{-h}^{0} 
\dfrac{1}{2}(1+e^{2s}) L_i(\gam(s),\dot \gam(s))\,ds 
+ \dfrac{1}{2}(1+e^{-2h})u_{i}(\gam(-h),t-h)\\
+
\int_{-h}^0 \dfrac{1}{2}(1-e^{2s}) L_j(\gam(s),\dot \gam(s))\,ds 
+ \dfrac{1}{2}(1-e^{-2h})u_{j}(\gam(-h),t-h)\
\mid \ \gam \in \AC([-h,0]),\ \gam(0)=x
\Big \}.
\end{multline}
\end{prop}

\begin{proof}
We denote by $w_i(x,t,h)$ the right hand side of \eqref{C-DPP}. 
For any $\gam \in \AC([-t,0])$ with $\gam(0)=x$,
set $\eta(s)=\gam(s-h)$ for $s\in [-t+h,0]$. 
Note that for $s<0$,
\[
\frac{1}{2}(1\pm e^{2(s-h)})
=
\frac{1}{2}(1+ e^{-2h})\cdot\frac{1}{2}(1\pm e^{2s})
+
\frac{1}{2}(1- e^{-2h})\cdot\frac{1}{2}(1\mp e^{2s}), 
\]
which actually comes from the 
memoryless property of 
Markov processes. We have 
\begin{align*}
&\int_{-t}^0 \dfrac{1}{2}(1+e^{2s}) L_i(\gam(s),\dot \gam(s))\,ds 
+ \dfrac{1}{2}(1+e^{-2t})g_{i}(\gam(-t))\\
&+
\int_{-t}^{0} \dfrac{1}{2}(1-e^{2s}) L_j(\gam(s),\dot \gam(s))\,ds 
+ \dfrac{1}{2}(1-e^{-2t})g_{j}(\gam(-t))\\
=&
\int_{-h}^0 \dfrac{1}{2}(1+e^{2s}) L_i(\gam(s),\dot \gam(s))\,ds
+\int_{-h}^0 \dfrac{1}{2}(1-e^{2s}) L_j(\gam(s),\dot \gam(s))\,ds \\
&+\dfrac{1}{2}(1+e^{-2h}) \Big [
\int_{-t+h}^{0} \dfrac{1}{2}(1+e^{2s}) L_i(\eta(s),\dot \eta(s))\,ds
+ \dfrac{1}{2}(1+e^{-2(t-h)})g_{i}(\eta(-t+h))\\
&+\int_{-t+h}^{0} \dfrac{1}{2}(1-e^{2s}) L_j(\eta(s),\dot \eta(s))\,ds
+ \dfrac{1}{2}(1-e^{-2(t-h)})g_{j}(\eta(-t+h))
\Big ]\\
&+\dfrac{1}{2}(1-e^{-2h}) \Big [
\int_{-t+h}^{0} \dfrac{1}{2}(1-e^{2s}) L_i(\eta(s),\dot \eta(s))\,ds
+ \dfrac{1}{2}(1-e^{-2(t-h)})g_{i}(\eta(-t+h))\\
&+\int_{-t+h}^{0} \dfrac{1}{2}(1+e^{2s}) L_j(\eta(s),\dot \eta(s))\,ds
+ \dfrac{1}{2}(1+e^{-2(t-h)})g_{j}(\eta(-t+h))
\Big ]\\
\ge  &
\int_{-h}^{0} \dfrac{1}{2}(1+e^{2s}) L_i(\gam(s),\dot \gam(s))\,ds 
+ \dfrac{1}{2}(1+e^{-2h})u_{i}(\gam(-h),t-h)\\
&+
\int_{-h}^{0} \dfrac{1}{2}(1-e^{2s}) L_j(\gam(s),\dot \gam(s))\,ds 
+ \dfrac{1}{2}(1-e^{-2h})u_{j}(\gam(-h),t-h),
\end{align*}
which implies $u_i(x,t) \ge w_i(x,t,h)$ for $i=1,2$.

We also can prove the other inequalities by a similar way. 
Thus, we omit the details. 
\end{proof}

\begin{prop}\label{prop:uni-conti}
The functions $u_i$ are continuous on $\cQ$. 
\end{prop}
\begin{proof}
We first prove that $u_i$ are Lipschitz continuous on $\cQ$ 
under the additional assumption that $g_i$ are Lipschitz continuous 
on $\T^n$. 
This additional requirement on $g_i$ will be removed at the end of 
the proof.

We may choose a constant $M_1>0$ so that 
$H_i(x,Dg_i(x))+(g_i-g_j)(x)\le M_1$ for a.e. $x\in\T^n$. 
It is clear that the function 
$v_i(x,t):=g_i(x)-M_1 t$ on $\cQ$ is a subsolution of 
(C). 

By a similar argument to the proof of Proposition \ref{prop:sub2} we obtain 
\[
v_i(x,t)\le \E_i\Big[
\int_{-t}^0L_{\nu(s)}(\gam(s),\dot\gam(s))\,ds
+v_{\nu(-t)}(\gam(-t),0)\Big]
\]
for all $(x,t)\in\cQ$ 
and $\gam\in\AC([-t,0])$ with $\gam(0)=x$, from which we get
$g_i(x)-M_1t\le u_i(x,t)$ for all $(x,t)\in\cQ$.

It follows from \eqref{C-rep} that
\[
u_i(x,t)\le
\E_i\Big[\int_{-t}^0L_{\nu(s)}(x,0)\,ds
+g_{\nu(-t)}(x)\Big]
\le 
g_i(x)+C_1t
\]
for
$C_1:=\max\{M_1, \max_{i=1,2;\, x\in \T^n} |L_i(x,0)|+\max_{x \in \T^n} |g_1(x)-g_2(x)| \}$ ,
and all $(x,t)\in\cQ$.  
Therefore we get 
\begin{equation}\label{conti-ini}
|u_i(x,t)-g_i(x)|\le C_1 t \ 
\text{for all} \ (x,t)\in\cQ.
\end{equation} 
Now, for any $(x,t) \in \Q$ and $h>0$, by Dynamic Programming Principle \eqref{C-DPP},
\begin{multline}\notag
u_i(x,t+h)=\inf \Big \{
\int_{-t}^{0} 
\dfrac{1}{2}(1+e^{2s}) L_i(\gam(s),\dot \gam(s))\,ds 
+ \dfrac{1}{2}(1+e^{-2t})u_{i}(\gam(-t),h)\\
+
\int_{-t}^0 \dfrac{1}{2}(1-e^{2s}) L_j(\gam(s),\dot \gam(s))\,ds 
+ \dfrac{1}{2}(1-e^{-2t})u_{j}(\gam(-t),h)\
\mid \ \gam \in \AC([-t,0]),\ \gam(0)=x
\Big \}.
\end{multline}
 By \eqref{conti-ini}, $|u_i(\gam(-t),h)-g_i(\gam(-t),h)| \le C_1 h$ for $i=1,2$. 
Hence, we derive that
\begin{equation}\label{Lip-t}
|u_i(x,t+h)-u_i(x,t)| \le C_1 h.
\end{equation}

We next prove that $u_i$ are Lipschitz {continuous} in $x$ for $i=1,2$.
Fix $x,y\in\T^n$ with $x\not=y$ and $t>0$. 
In view of {the coercivity} of $H_i$, there exists a constant $\rho>0$ 
such that $L_i(x, \xi)\le C$ for all $x \in\T^n$ and $\xi \in B(0, \rho)$ 
(see \cite[Proposition 2.1]{I2008}).
Set $\tau:=|x-y|/\rho$ and 
we first consider the case where $\tau<t$. 
Set
$\eta(s):=y-s\rho(x-y)/|x-y|$ for $s\in[-\tau,0]$. 
Note that $\eta\in \AC([-t,0])$, $\eta(0)=y$ and $\eta(-\tau)=x$.
By Dynamic Programming Principle \eqref{C-DPP}, \eqref{conti-ini} and \eqref{Lip-t},
\begin{align*}
u_i(y,t)\le&\, 
\E_i\Big[
\int_{-\tau}^0 L_{\nu(s)}(\eta(s),\dot \eta(s))\,ds 
+ u_{\nu(-\tau)}(\eta(-\tau),t-\tau)\Big]\\
\le&\,
C\tau+
\dfrac{1}{2}(1+e^{-2\tau}) u_i(x,t-\tau)+
\dfrac{1}{2}(1-e^{-2\tau}) u_j(x,t-\tau)\\
\le&\,
(C+2C_1t)\tau + u_i(x,t) \le C |x-y| + u_i(x,t),
\end{align*}

By symmetry we conclude 
$|u_{i}(x,t)-u_{i}(y,t)|\le C|x-y|$, where $C$ depends on $t$ as calculated above.
Notice that this is just a fairly crude estimate, but it is good enough for our presentation here.

We consider the case where $t\le\tau$. 
By \eqref{conti-ini}, 
\begin{eqnarray*}
|u_i(x,t)-u_i(y,t)|
&\le&|u_i(x,t)-g_{i}(x)|+|g_{i}(x)-g_{i}(y)|+|g_{i}(y)-u_i(y,t)|\\
&\le&2Ct+M|x-y|\le C|x-y|,
\end{eqnarray*}
where $M:=\max_{i=1,2}\|Dg_{i}\|_{L^\infty(\T^n)}$. 
Thus, we get 
$|u_i(x,t)-u_i(y,t)|\le C|x-y|$ 
for all $x,y\in \T^n, t\ge 0$ and $i=1,2$. 

We finally remark that 
we can deduce the continuity of $u_i$ 
by using an approximation argument.  
We may choose a sequence 
$\{g_{i}^{k}\}_{k\in\N}$ of Lipschitz continuous functions  
so that $\|g_{i}^{k}-g_{i}\|_{\Li(\T^n)} \le 1/k$ 
for all $k\in\N$. 
Let $u_{i}^{k}$ be functions defined by \eqref{C-rep} with given $g_i^k$. 
By comparison through the formulas
for $u_{i}$ and $u_{i}^{k}$, we see that 
$|u_{i}(x,t)-u_{i}^{k}(x,t)|\le \max_{\T^n}|g_i-g_i^k|$.
Since $u_{i}^{k}\in {\rm C}(\cQ)$ by the above argument
for all $k\in\N$ and  
$u_{i}^k$ converges uniformly to $u_i$ on $\cQ$,  
we obtain $u_i\in {\rm C}(\cQ)$. 
\end{proof}

\begin{proof}[Proof of Theorem {\rm \ref{thm:representation}}]
It is clear that $(u_1,u_2)(\cdot,0)=(g_{1},g_{2})$ on $\T^n$.
We now prove that $u_1$ is a subsolution of (C). Take a test function
$\phi \in C^1(\Q)$ such that $u_1-\phi$ has a maximum at $(x_0,t_0) \in \Q$
and $(u_1-\phi)(x_0,t_0)=0$.
Take $h>0$ small enough. By Proposition \ref{prop:DPP},
\begin{multline*}
u_1(x_0,t_0) \le 
\int_{-h}^0 \dfrac{1}{2}(1+e^{2s}) L_1(\gam(s),\dot \gam(s))\,ds + \dfrac{1}{2}(1+e^{-2h})u_{1}(\gam(-h),t_0-h)\\
+
\int_{-h}^0 \dfrac{1}{2}(1-e^{2s}) L_2(\gam(s),\dot \gam(s))\,ds + \dfrac{1}{2}(1-e^{-2h})u_{2}(\gam(-h),t_0-h)
\end{multline*}
for any $\gam \in \AC([-h,0])$ with $\gam(0)=x_0$ and $\dot \gam(0)=q \in \R^n$.
We now use the fact $u_1(\gam(-h),t_0-h) \le \phi(\gam(-h),t_0-h)$ to plug into the above to derive that
\begin{multline*}
\dfrac{\phi(\gam(0),t_0)-\phi(\gam(-h),t_0-h)}{h} 
\le
\dfrac{1}{h} \int_{-h}^{0} \dfrac{1}{2}(1+e^{2s}) L_1(\gam(s),\dot \gam(s))\,ds\\
+\dfrac{1}{h}\int_{-h}^{0} \dfrac{1}{2}(1-e^{2s}) L_2(\gam(s),\dot \gam(s))\,ds+
 \dfrac{1-e^{-2h}}{2h}(u_{2}-u_1)(\gam(-h),t_0-h).
\end{multline*}
Sending $h\to 0$, we obtain
$$
\phi_t(x_0,t_0)+D\phi(x_0,t_0) \cdot q \le L_1(x_0,q)+(u_2-u_1)(x_0,t_0) \
\text{for all} \ q \in \R^n,
$$
which implies 
$\phi_t(x_0,t_0)+H_1(x_0,D\phi(x_0,t_0))+(u_1-u_2)(x_0,t_0) \le 0$.

Next we prove that $u_1$ is a supersolution of (C). Take a test function
$\phi \in C^1(\Q)$ such that $u_1-\phi$ has a minimum at $(x_0,t_0)\in \Q$
and $(u_1-\phi)(x_0,t_0)=0$.
Take $h>0$ small enough. By Proposition \ref{prop:DPP},
there exists $\gam_h \in \AC([-h,0])$ with $\gam_h(0)=x_0$ such that
\begin{multline*}
u_1(x_0,t_0) +h^2 
\ge \int_{-h}^{0} \dfrac{1}{2}(1+e^{2s}) L_1(\gam_h(s),\dot \gam_h(s))\,ds 
+ \dfrac{1}{2}(1+e^{-2h})u_{1}(\gam_h(-h),t_0-h)\\
+
\int_{-h}^{0} \dfrac{1}{2}(1-e^{2s}) L_2(\gam_h(s),\dot \gam_h(s))\,ds + \dfrac{1}{2}(1-e^{-2h})u_{2}(\gam_h(-h),t_0-h).
\end{multline*}
We use $u_1(\gam_h(-h),t_0-h) \ge \phi(\gam_h(-h),t_0-h)$ in the above to yield
\begin{multline}\label{veri-4}
\dfrac{\phi(\gam_h(0),t_0)-\phi(\gam_h(-h),t_0-h)}{h} +h \ge 
\dfrac{1}{h} \int_{-h}^{0} \dfrac{1}{2}(1+e^{2s}) 
L_1(\gam_h(s),\dot \gam_h(s))\,ds\\
+\dfrac{1}{h}\int_{-h}^{0} \dfrac{1}{2}(1-e^{2s}) 
L_2(\gam_h(s),\dot \gam_h(s))\,ds+
 \dfrac{1-e^{-2h}}{2h}(u_{2}-u_1)(\gam(-h),t_0-h).
\end{multline}
On the other hand,
\begin{align}
&\dfrac{\phi(\gam_h(0),t_0)-\phi(\gam_h(-h),t_0-h)}{h} \nonumber\\
=&\,
\dfrac{1}{h} \int_{-h}^{0} 
\phi_t(\gam_h(s),t_0-s)+D\phi(\gam_h(s),t_0-s)\cdot \dot \gam_h(s)\,ds\nonumber\\
\le&\, 
\dfrac{1}{h} \int_{-h}^{0} \phi_t(\gam_h(s),t_0-s)\,ds\nonumber\\
&\, 
+\dfrac{1}{h} \int_{-h}^{0} \dfrac{1}{2}(1+e^{2s}) 
\Big\{
L_1(\gam_h(s),\dot \gam_h(s))+H_1(\gam_h(s),D\phi(\gam_h(s),t_0-s))
\Big\}\,ds\nonumber\\
&+\dfrac{1}{h}\int_{-h}^{0} \dfrac{1}{2}(1-e^{2s}) 
\Big\{ L_2(\gam_h(s),\dot \gam_h(s))+H_2(\gam_h(s),D\phi(\gam_h(s),t_0-s))\Big\}\,ds.
\label{veri-5}
\end{align}
Combine \eqref{veri-4}, \eqref{veri-5} and then send $h\to 0$ to get 
\[
\phi_t(x_0,t_0)+H_1(x_0,D\phi(x_0,t_0))+(u_1-u_2)(x_0,t_0) \ge 0. 
\]
{It is easy to see the uniform continuity of $u_i$ due to the coercivity 
of Hamiltonians.}
\end{proof}

\noindent
{\bf Acknowledgement.} 
The authors are grateful to Professor Toshio Mikami for discussions 
which has been of help for them to come to Theorem \ref{thm:representation}. 
We also thank Professor Naoyuki Ichihara for his useful comments on Lemma \ref{lem:half1} 
and Professor Andrea Davini for pointing out a gap in the first preprint.
The work of H. M. was partly supported by by Grant-in-Aid for Scientific Research, No.24840043 (Research Activity Start-up) and Grant for Basic Science Research Projects from the 
Sumitomo Foundation. 

\bibliographystyle{amsplain}

\begin{thebibliography}{1}

\bibitem{BS}
G. Barles and P. E. Souganidis, \emph{On the large time behavior of solutions of Hamilton--Jacobi equations}, 
SIAM J. Math. Anal. {\bf 31} (2000), no. 4, 925--939.

\bibitem{BJ}
E. N. Barron and R. Jensen, 
\emph{Semicontinuous viscosity solutions for Hamilton-Jacobi equations with
convex Hamiltonians}, 
Comm. Partial Differential Equations {\bf 15} (1990), no. 12, 1713--1742.


\bibitem{CGMT1}
F. Cagnetti, D. Gomes, H. Mitake and H. V. Tran, 
\emph{A new method for large time behavior of convex Hamilton--Jacobi equations I: Degenerate equations and weakly coupled systems}, arXiv:1212.4694, 2012.

\bibitem{CGT2}
F. Cagnetti, D. Gomes and H. V. Tran, 
\emph{Adjoint methods for obstacle problems and weakly coupled systems of {P}{D}{E}}, to appear in ESAIM Control Optim. Calc. Var. 

\bibitem{CLL}
F. Camilli, O. Ley  and P. Loreti,
\emph{Homogenization of monotone
systems of Hamilton-Jacobi equations}, 
ESAIM Control Optim. Calc. Var. {\bf 16} (2010), 58-76.

\bibitem{CLLN}
F. Camilli, O. Ley, P. Loreti and V. Nguyen, 
\emph{Large time behavior of weakly coupled systems of first-order Hamilton--Jacobi equations}, to appear in NoDEA. 


\bibitem{DS}
A. Davini and A. Siconolfi, 
\emph{A generalized dynamical approach to the large-time behavior of solutions of Hamilton--Jacobi equations}, 
SIAM J. Math. Anal. {\bf 38} (2006), no. 2, 478--502.  

\bibitem{DZ}
A. Davini, M. Zavidovique, 
\emph{Aubry sets for weakly coupled
systems of Hamilton--Jacobi equations}, arXiv:1211.1245, 2012.

\bibitem{EL}
H. Engler, S. M. Lenhart, 
\emph{Viscosity solutions for weakly coupled systems of Hamilton--Jacobi equations}, 
Proc. London Math. Soc. (3) {\bf 63} (1991), no. 1, 212--240. 

\bibitem{F2}
A. Fathi, 
\emph{Sur la convergence du semi-groupe de Lax-Oleinik}, 
C. R. Acad. Sci. Paris S\'er. I Math. {\bf 327} (1998), no. 3, 267--270.

\bibitem{I2008}
H. Ishii, 
\emph{Asymptotic solutions for large-time of Hamilton--Jacobi equations in Euclidean n space}, 
Ann. Inst. H. Poincar\'e Anal. Non Lin\'eaire, {\bf25} (2008), no 2, 231--266. 

\bibitem{IK}
H. Ishii and S. Koike, 
\emph{Viscosity solutions for monotone systems of second-order elliptic {PDE}s}, 
Comm. Partial Differential Equations {\bf 16} (1991), no. 6-7, 1095--1128. 

\bibitem{LPV}  
P.-L. Lions, G. Papanicolaou and S. R. S. Varadhan,  
\emph{Homogenization of Hamilton--Jacobi equations}, unpublished work (1987). 

\bibitem{MT1}
H. Mitake, H. V. Tran, 
\emph{Remarks on the large-time behavior of  
viscosity solutions of quasi-monotone weakly coupled systems 
of Hamilton--Jacobi equations}, 
Asymptot. Anal., 
{\bf77} (2012), 43--70.  

\bibitem{MT2}
H. Mitake, H. V. Tran, 
\emph{Homogenization of weakly coupled systems of
Hamilton--Jacobi equations with fast switching rates}, 
submitted. 

\bibitem{NR}
G. Namah and J.-M. Roquejoffre, 
\emph{Remarks on the long time behaviour of the solutions of Hamilton--Jacobi equations}, 
Comm. Partial Differential Equations {\bf 24} (1999), no. 5-6, 883--893. 

\bibitem{VN}
V. D. Nguyen,
\emph{Some results on the large time behavior of weakly coupled
systems of first-order Hamilton-Jacobi equations}, arXiv:1209.5929, 2012.

\bibitem{R}
J.-M. Roquejoffre, 
\emph{Convergence to steady states or periodic solutions in a class of Hamilton--Jacobi equations}, 
J. Math. Pures Appl. (9) {\bf 80} (2001), no. 1, 85--104. 


\end{thebibliography}
\providecommand{\bysame}{\leavevmode\hbox to3em{\hrulefill}\thinspace}
\providecommand{\MR}{\relax\ifhmode\unskip\space\fi MR }
\providecommand{\MRhref}[2]{%
  \href{http://www.ams.org/mathscinet-getitem?mr=#1}{#2}
}
\providecommand{\href}[2]{#2}

\end{document}